\RequirePackage[l2tabu, orthodox]{nag}
\documentclass[a4paper, 11pt]{article}
\usepackage[utf8]{inputenc}
\usepackage[T1]{fontenc}
\usepackage{xspace}
\usepackage{mathtools}
\usepackage{bm}
\usepackage[final]{showkeys}
\usepackage{microtype}
\usepackage[title]{appendix}

\usepackage[noTeX]{mmap}
\usepackage[a4paper,left=1.35in, right=1.35in, bottom = 1.4in, top = 1.4in]{geometry}
\usepackage{amsmath,amsfonts, amsthm, amssymb}
\usepackage{color}
\usepackage{enumerate}
\usepackage{hypbmsec}
\usepackage{bbm}
\usepackage{dsfont}
\usepackage{graphicx}
\usepackage{caption}
\usepackage[section]{placeins}
\usepackage{tikz}
\usetikzlibrary{calc}
\usetikzlibrary{shapes}
\usetikzlibrary{decorations.pathreplacing}
\usetikzlibrary{arrows}
\usetikzlibrary{matrix}
\usetikzlibrary{fit}
\usetikzlibrary{positioning}
\usepackage{pgffor}
\usepackage{comment}

\usepackage{fancyhdr}
\usepackage{hyperref} 
\hypersetup{colorlinks}

\theoremstyle{plain}
\newtheorem{theorem}{Theorem}[section]

\newtheorem{corollary}[theorem]{Corollary}
\newtheorem{lemma}[theorem]{Lemma}
\newtheorem{definition}[theorem]{Definition}

\theoremstyle{definition}
\newtheorem{remark}[theorem]{Remark}

\DeclarePairedDelimiter\abs{\lvert}{\rvert}%
\DeclarePairedDelimiter\norm{\lVert}{\rVert}%

\newcommand{\RR}{\mathbb{R}}
\newcommand{\B}{\mathcal{B}}
\newcommand{\NN}{\mathbb{N}}
\newcommand{\ZZ}{\mathbb{Z}}

\newcommand{\E}{\mathbb{E}}

\newcommand{\Pb}{\mathbb{P}}

\newcommand{\indicator}{\mathds{1}}

\newcommand{\supp}{\operatorname{supp}}
\renewcommand{\d}{\mathsf{d}}

\newcommand{\pref}[1]{\hyperref[#1]{[p. \pageref{#1}]}}

\newcommand{\old}[1]{}

\newcommand{\SG}{\mbox{\textsf{SG}}\xspace}

\usepackage{color,array}
\def\zerotest{0}
\newcommand\greytest[2]{%
	\def\test{#2}%
	\ifx\test\zerotest
	\def\next{\textcolor[gray]{0.6}{0}}%
	\else
	\def\next{#1#2}%
	\fi
	\next}

\title{Divisible sandpile on Sierpinski gasket graphs}

\author{Wilfried Huss and Ecaterina Sava-Huss}

\date{\today}

\begin{document}

\maketitle

\begin{abstract}
The {\it divisible sandpile model} is a growth model on graphs that was
introduced by \textsc{Levine and Peres}  \cite{levine_peres_2009} as a
tool to study internal diffusion limited aggregation.
In this work we investigate the shape of the {\it divisible sandpile
model} on the graphical Sierpinski gasket $\SG$. We show that the shape
is a ball in the graph metric of $\SG$. Moreover we give an exact
representation of the odometer function of the divisible sandpile.
\end{abstract}



\section{Introduction}\label{sec:intro}

The \emph{divisible sandpile} model was introduced by
\textsc{Levine and Peres} \cite{levine_peres_2009} as a tool
to study growth models such as \emph{internal diffusion limited aggregation} and \emph{rotor-router aggregation}. In the model
every vertex of a graph contains a certain mass of sand. If the
mass at a vertex exceeds a certain value (such a vertex is called unstable), the vertex is stabilized by distributing the excess mass uniformly among the neighbors of the vertex. The process continues as long as there are unstable vertices. We are interested in the set of vertices that have positive mass in the limit configuration when the process starts with
a big amount of mass at one vertex of the graph. Such a set is called the
\emph{divisible sandpile cluster}.
The limit shape of the divisible sandpile cluster was identified on $\mathbb{Z}^d$ in \cite{levine_peres_2009}, on homogeneous trees in \cite{Levine-sandpile-tree}, and on the comb lattice in \cite{idla-comb-huss-sava}. See also the recent survey \cite{levine_peres_2016} for an
introduction to the divisible sandpile model.

The aim of this paper is to identify the limit shape of the divisible sandpile cluster on the doubly-infinite Sierpinski gasket graph $\SG$, by making strong use of the property of $\SG$ of being finitely ramified, which  means that  it can be disconnected by removing a finite number of points. On the same graph, by using the limit shape of the divisible sandpile cluster, we prove in \cite{idla-sg-chen-huss-sava-teplyaev} a limit shape theorem for the internal diffusion limited aggregation.

\emph{The Sierpinski gasket graph} is a pre-fractal associated with the Sierpinski gasket, defined as following.
Given a subset $S\subset\ZZ^2$ and a function $\varphi:\ZZ^2\to\ZZ^2$ define
$\varphi(S) = \{\varphi(z):\: z\in S\}$. Let $G_0$
be the complete graph on the three given vertices $ \{(0,0),(1,0),(0,1)\}$ in $\ZZ^2$. Recursively
given a graph $G_k$ define its next iteration
\begin{equation*}
G_{k+1} = \bigcup_{i=0}^2 \varphi_{k,i}(G_k),
\end{equation*}
where $\varphi_{k,i}(x,y) = (x,y) + a_{k,i}$ with
$a_{k,0} = (0,0)$, $a_{k,1} = (2^k,0)$ and $a_{k,2} = (0,2^k)$.
The one-sided graphical Sierpinski gasket $\SG^+$ is then defined as
\begin{equation*}
\SG^+ = \bigcup_{k\geq 0} G_k.
\end{equation*}
Denote by $\SG^- = -(\SG^+)$ its mirror image. The double-sided graphical Sierpinski gasket $\SG$ is then defined as $\SG = \SG^+ \cup \SG^-$. In
the remainder of the paper we will call $\SG$ the {\it Sierpinski gasket} or the
{\it Sierpinski gasket graph} for simplicity.
We denote the neighborhood relation in $\SG$ by $\sim$. Note
that $\SG$ is a $4$-regular graph, and the vertex set of $\SG$ is a subset of the two dimensional integer  lattice $\ZZ^2$. This definition is convenient for our use, since it allows
us to specify vertices of $\SG$ simply by their rectangular coordinates.
Moreover, functions on $\SG$ will be denoted as functions $\ZZ^2\to\RR$ restricted to $\SG$. For the drawings we will use the more common
planar embedding given by the function
\begin{equation*}
\psi(x,y) = \left(x+\frac{1}{2}y,\frac{\sqrt{3}}{2}\abs{y}\right),
\end{equation*}
see Figure \ref{fig:embedding}.
\begin{figure}[t]
	\centering
	\input{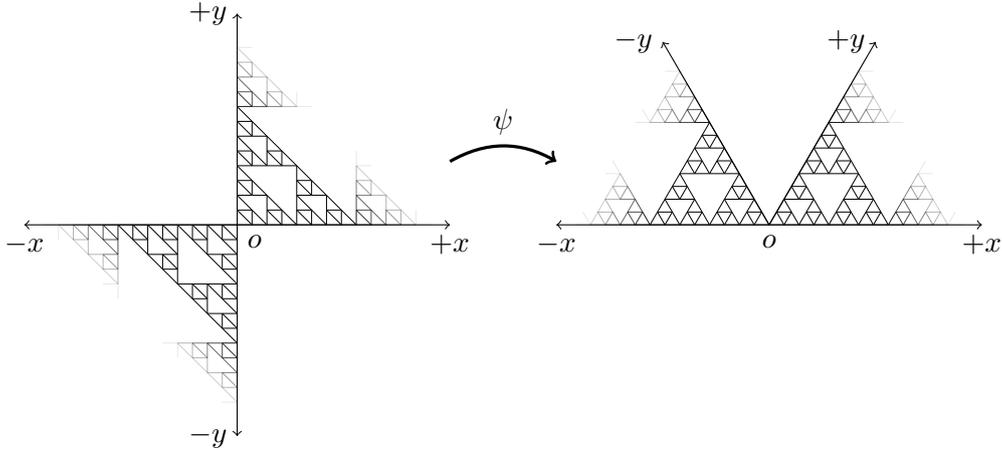}
	\caption{\label{fig:embedding} The embedding used to draw \SG.}
\end{figure}



Our main result is the following shape theorem for the divisible sandpile model on $\SG$. Denote by $B_n$ the ball of radius $n$ and center $o=(0,0)$ in the graph metric of $\SG$, and by $b_n:=|B_n|-1/2|\partial B_n|$, with $\partial B_n=\{u\in B_n:\ \exists v\in B_n^c\ \text{ with } u\sim v\}$.
\begin{theorem}
\label{thm:sandpile_shape_thm}
For any $m\geq 0$, let $n = \max\{k \geq 0: b_k \leq m \}$. If  $S_m$ is the divisible sandpile cluster on $\SG$ with the initial mass configuration $\mu_0 = m\delta_o$,  then
\begin{align*}
B_{n-1} \subseteq S_m \subseteq B_n.
\end{align*}
\end{theorem}
The paper is organized as follows. In Section \ref{sec:prelim} we introduce
the necessary notions on Sierpinski gasket graphs and some basic facts
about random walks and Green functions. Subsequently, in Section \ref{sec:divisible_sandpile} we formally define the divisible sandpile model. Section \ref{sec:sierpinski_gasket} is devoted to
the proof of Theorem \ref{thm:odometer_bn} which describes the limit shape of the divisible sandpile cluster with initial mass $b_n$ at the origin.
The main Theorem \ref{thm:sandpile_shape_thm} is then an easy consequence of Theorem \ref{thm:odometer_bn}. In Section \ref{sec:divisible_sandpile} and \ref{sec:sierpinski_gasket} we assume the existence of a function with Laplacian equal to $1$ on the whole graph. In Section \ref{sec:constant_laplacian} we give an explicit construction of such a function, with particularly nice properties, and we show the connection between this function and the odometer function of the divisible sandpile model on $\SG$. Then in Theorem \ref{thm:odometer_V_k} we give an explicit
construction of the odometer function for the divisbile sandpile
with initial mass of $3^{k+1}$ at the origin. In the explicit construction of a function with Laplacian $1$ on the Sierpinski gasket, we made use of a {\it generalized $\frac{1}{5}-\frac{2}{5}$} rule,  which will be proved in Appendix A. We conclude the paper with some questions.

\section{Preliminaries}\label{sec:prelim}

\subsection{The graphical Sierpinski gasket}

Let \SG be the Sierpinski gasket graph  as defined in the Introduction, and denote the neighborhood relation in \SG by $\sim$.
Recall that $\SG^+ = \SG \cap [0,\infty)^2$ (resp. $\SG^- = \SG \cap (-\infty,0]^2$) denotes the positive (resp. negative) branch of the Sierpinski gasket graph. For any subset $A\subset \SG$ write $A^+ = A\cap \SG^+$ and $A^- = A\cap \SG^-$. 
We denote the graph metric in \SG by $\d$, that is 
for vertices $u,v\in\SG$, $\d(u,v)$ is the length of the shortest path from $u$ to $v$.
Note that if $u = (x,y)\in\SG$ the distance to the origin is given
by $\d(o, u) = \abs{x} + \abs{y}$.
The ball of radius $n$ in the graph distance of $\SG$ around the origin is given by
$$B_n = \big\{(x,y)\in \SG: \abs{x} + \abs{y} \leq n\big\}.$$
For $k\geq 0$ denote by $V_k = B_{2^k}$ the {\it $k$-th full iteration in $\SG$}. The extremal points of $V_k$ are denoted by $\partial V_k = \big\{(2^k,0), (0,2^k), (-2^k,0), (0,-2^k)\big\}$. 

For any any set $A\subset\SG$, denote by $\partial A = \{u \in A: \exists v\not\in A \text{ s.t. } u\sim v\}$ {\it the inner boundary of $A$}, while
$\partial_\circ A = \{u \in\SG: u \not\in A \text{ and } \exists v\in A \text{ s.t } u\sim v\}$ denotes the {\it outer boundary}.
Denote
by  $S_n = \{u \in \SG: d(o,u) = n\}$ the {\it sphere of radius $n$}, and 
by $\partial_I B_n = S_n \setminus \partial B_n$ be the set
of points of the sphere or radius $n$ which have no neighbor outside the ball with the same radius.

Let $f:\SG\to\RR$ be a real valued function on \SG, then the operator
\begin{equation*}
\Delta f(x) = \frac{1}{4} \sum_{y\sim x} f(y) - f(x),
\end{equation*}
defines the \emph{discrete graph Laplacian} of $f$. If $\Delta f(x) =0$, then $f$ is called \emph{harmonic}, and if $\Delta f(x) \geq 0$ (respectively $\Delta f(x) \leq 0$ ), then $f$ is called \emph{subharmonic} (respectively \emph{super-harmonic}).

\subsection{Green function and random walks}

The (discrete time) simple random walk (SRW) $\big(X(t)\big)_{t\geq 0}$ on $\SG$ is the (time homogeneous) 
Markov chain with one-step transition probabilities given by
\begin{equation*}
p(x,y):=\mathbb{P}[X(t+1)=y\mid X(t)=x]=\frac{1}{4}
\end{equation*}
if $y\sim x$, and $0$ otherwise. 
We denote by $\mathbb{P}_x$ and $\mathbb{E}_x$ the probability law and the expectation
of the random walk $X(t)$ starting at $x\in \SG$. For a
finite subset $A\subset \SG$ be denote by $g_A$ the Green function stopped at the set $A$. That is, if 
$$\tau_A = \inf\{t: X(t) \not\in A \}$$
is the first exit time
of $A$, then the \emph{stopped Green function} is defined as
\begin{align*}
g_A(x,y) = \E_x\left[\sum_{t=0}^{\tau_A - 1} \indicator_{\{X(t)=y\}}\right].
\end{align*}
The stopped Green function represents the expected number of visits to $y$ before exiting the set $A$, with the random walk starting at $x$.
The \emph{harmonic measure} of the set $A$ is then defined as
\begin{align*}
\nu(x) = \Pb_o\big[X(\tau_A) = x\big].
\end{align*}

\section{The divisible sandpile}\label{sec:divisible_sandpile}

In this section we formally define the {\it divisible sandpile model} on
$\SG$. We will mostly follow the notation of \cite{levine_peres_2009} where
the divisible sandpile model was originally introduced in the case of the
Euclidean lattice $\ZZ^d$. We give the full definition and will state
the main convergence results for the divisible sandpile to make the
presentation more self contained. We will need a slightly more
general version of the divisible sandpile as the one in
\cite{levine_peres_2009}. While all results of this section can be proven
on any locally finite graph, which admits an irreducible reversible Markov transition operator, for simplicity we will define the model only on the Sierpinski gasket graph $\SG$.

Fix a function $h:\SG\to [0,\infty]$, which describes the maximal
height of the sandpile at any vertex. We have to assume that $\sum_{z\in\SG} h(z) = \infty$ in order to ensure that the sandpile cluster, which will be defined in Definition \ref{def:sandpile_cluster}, is always a finite set.

\begin{remark}
If not specified otherwise, we will always let $h$ to be the constant function $1$. For the special case $h\equiv 1$ we recover the model as
defined in \cite{levine_peres_2009}.
\end{remark}
We call a function $\mu:\SG\to\RR_{\geq 0}$ with finite support
$\abs{\supp(\mu)}<\infty$ a \emph{sand distribution} on \SG.
Given a sand distribution $\mu$ and a vertex $x\in \SG$, the \emph{toppling operator} is defined as
\begin{equation*}
T_x \mu = \mu + \max \{ \mu(x)-h(x), 0 \} \Delta\delta_x.
\end{equation*}
The toppling operator $T_x$ affects the sand distribution as follows:
if the sandpile at $x$ exceeds the threshold $h(x)$, that is,
if $\mu(x) > h(x)$ the excess mass $\mu(x)-h(x)$ is distributed
equally among the neighbors of $x$. On the other hand, if the sandpile
at $x$ is smaller than the threshold, the sand distribution remains
unchanged.

Let now \(\mu_0\) be an {\it initial sand distribution} on $\SG$, and
$\big(x_k\big) _{k\geq 1}$ be a sequence of vertices in $\SG$ called
the \emph{toppling sequence}, with the property that $\big(x_k\big)$
contains each vertex of $\SG$ infinitely often.
We define the sand distribution of the sandpile after $k$
steps recursively as
\begin{equation*}
\mu_{k}(y) = T_{x_k}\mu_{k-1}(y) = T_{x_k}\cdots T_{x_1}\mu_0(y),
\end{equation*}
where $y\in\SG$.
The sand distribution $\mu_k$ represents the amount of mass at each
vertex of $\SG$ after the successive toppling of the vertices
$x_1,\ldots,x_k$. Denote by 
\begin{equation}
\label{eq:total_mass}
M = \sum_{x\in \SG} \mu_0(x)
\end{equation}
the total mass of the sandpile.
Note that by construction $M = \sum_{x\in SG} \mu_k(x)$,
for all $k\geq 0$. In other words, the total mass $M$ is conserved during the whole process, it just gets redistributed.
One important tool that will be used throughout this work is the
so-called \emph{odometer function} of the divisible sandpile,  introduced in \textsc{Levine and Peres} \cite{levine_peres_2009}.
\begin{definition}
	The \emph{odometer function} after $k$ topplings $u_k$ is defined as
	\begin{equation*}
	u_k(y) =  \sum_{j\leq k:\:x_j=y} \mu_j(y)-\mu_{j+1}(y),
	\end{equation*}
	and represents the total mass emitted from a vertex $y\in \SG$ during the first $k$ topplings.
\end{definition}

\subsection{Convergence of the divisible Sandpile}
\label{subsec:sandpile_conv}
We list here the relevant convergence results for the
divisible sandpile, whose proofs in the case of
Euclidean lattices can be found in $\mathbb{Z}^d$ \cite{levine_peres_2009}. 
The proofs work the same way on any regular graph $G$, as long as there
exists a function $\ell: G\to\RR$ with globally constant Laplacian, i.e.
$\Delta \ell(z) = 1$, for all $z\in G$. In the case of $\mathbb{Z}^d$
one can use the function $\ell(z) = \norm{z}^2$. 
On Cayley graphs of finitely generated groups, the existence
of a function with constant Laplacian on the whole graph follows from
a theorem of \textsc{Ceccherini-Silberstein and Coornaert}
\cite{ceccherini_coornaert}. We will construct such a function on $\SG$ in Section \ref{sec:constant_laplacian}.

In order to prove that the sequence of mass distributions $(\mu_k)_{k\geq 1}$ has a limit, one first
proves that the sequence of odometer functions $(u_k)_{k\geq 1}$ converges.
For a proof of the next lemma see \cite[Lemma 3.1]{levine_peres_2009}.
\begin{lemma}\label{lemma:convergence_odometer}
As $k\to\infty$, the sequence of functions $(u_k)_{k\geq 1}$ and the sequence of sand distributions $(\mu_k)_{k\geq 1}$ converge point-wise to limit functions $u_k \nearrow u$ and $\mu_k \to \mu$. Moreover, the limit functions $\mu$ and $u$ satisfy the following relation
\begin{equation*}
\mu(z) = \mu_0(z) + \Delta u(z) \quad \text{and}\quad \mu(z) \leq h(z),\quad \text{for all } z \in \SG.
\end{equation*}
\end{lemma}

\begin{definition}
\label{def:sandpile_cluster}
We call $u$ the \emph{odometer function} of the divisible sandpile.
The set $\mathcal{S} = \{z\in\SG: u(z) > 0\}$ is called the \emph{divisible sandpile cluster}, or the {\emph sandpile cluster} for short.
\end{definition}

\begin{remark}
By construction $\mu(z) = h(z)$ for all $z\in\mathcal{S}$. It follows that $\mathcal{S}$ is a finite set, since by assumption $\sum_{z\in\SG} h(z) = \infty$.
\end{remark}

\subsection{Abelian Property}

Everything we did until now depends on the chosen toppling sequence
$(x_k)_{k\geq 0}$. In the next Lemma we prove the 
\emph{Abelian property}.

\begin{lemma}[Abelian Property]
\label{lem:abelian_property}
The odometer function $u$ is independent of the choice of the toppling sequence.
\end{lemma}
\begin{proof}
	Assume that there are two toppling sequences that result in different limits $u_1$ and $u_2$
	of the odometer function. Denote by 
	\begin{equation*}
	\mu_1 = \mu_0 + \Delta u_1 \text{ resp. } 
	\mu_2 = \mu_0 + \Delta u_2,
	\end{equation*}
	the resulting sand distributions, and by 
	\begin{equation*}
	S_1 = \{z\in \SG: u_1(z) > 0 \} \text{ resp. } S_2 = \{z\in \SG: u_2(z) > 0 \}
	\end{equation*}
the sets of vertices that toppled in each of the two toppling sequences.
Consider the set $\mathcal{A} = \{z\in \SG: u_1(z) > u_2(z) \}$. Since $u_2(z) \geq 0$ we have $\mathcal{A} \subset S_1$. In particular $\mathcal{A}$ is finite.
	Assume $\mathcal{A}$ is not empty. By construction $\mu_1(z) = h(z)$ for all $z\in\mathcal{A}$. By Lemma \ref{lemma:convergence_odometer}, $\mu_2(z) \leq h(z)$ for all $z\in \SG$, which implies that
 for all $z\in \mathcal{A}$
	\begin{equation}
	\label{eq:least_action_a_1}
	\mu_1(z) - \mu_2(z) \geq 0.
	\end{equation}
	
	Together with Lemma \ref{lemma:convergence_odometer} this yields
	\begin{align*}
	\Delta(u_1 - u_2)(z) = \big(\mu_1(z) - \mu_2(z)\big) \geq 0,
	\end{align*}
	for all $z\in \mathcal{A}$.
	Thus the function $f(z) = u_1(z) - u_2(z)$ is subharmonic on $\mathcal{A}$. Moreover $f(z) > 0$ for all $z\in \mathcal{A}$ and $f(z)\leq 0$ for all $z\not\in \mathcal{A}$.
	This implies that $f$ attains its maximum in the set $\mathcal{A}$. By the maximum
	principle for subharmonic functions (see i.e. \cite[Proposition 1.4]{kumagai_saintflour}), it follows that $f$ is constant on $A\cup\partial_\circ A$, which is a contradiction. Thus $\mathcal{A}$ is empty, and $u_1 \leq u_2$. 
Reversing the roles of $u_1$ and $u_2$ finishes the argument.
\end{proof}

\begin{remark}
	\label{rem_automorphism}
	A consequence of the Abelian property is that $u$, $\mu$ and $S$
	are invariant under all automorphisms of the graph $\SG$
	which fix the start distribution $\mu_0$.
\end{remark}

The next Lemma  provides a way to actually
compute the odometer function as the solution of a discrete obstacle
problem, see  \cite[Lemma 3.2]{levine_peres_2009}. We first
introduce some additional concepts.

\begin{definition}
	\label{def:superharmonic_majorant}
	Let $g: \SG\rightarrow \RR$ be a function on $\SG$. Define
	its \emph{least super-harmonic majorant} on a finite set
	$\mathfrak{B} \subset \SG$ as:
	\begin{equation*}
	s_g^\mathfrak{B}(z) = \inf \big\lbrace f(z):\: f \text{ super-harmonic on } \mathfrak{B},\, f \geq g \big\rbrace.
	\end{equation*}
\end{definition}

Remark that the function $s_g^\mathfrak{B}$ is itself super-harmonic on $\mathfrak{B}$. From Lemma \ref{lemma:convergence_odometer} we get 
\begin{equation*}
\Delta u(z) = \mu(z) - \mu_0(z) \leq h(z) - \mu_0(z).
\end{equation*}
In particular, if $z$ is an element of the sandpile cluster $\mathcal{S}$ we have
\begin{equation}
\label{eq:sandpile_odometer_laplace}
\Delta u(z) = h(z) - \mu_0(z).
\end{equation}

Let $\mathfrak{B} = \big\{z\in\SG: \d(z, \supp \mu_0) \leq M \big\}$,
where $M$ is the total mass of the sandpile as defined in
\eqref{eq:total_mass}. Then trivially $\mathcal{S} \subset \mathfrak{B}$.

 Define the function $\gamma:\SG\rightarrow \RR$ as
\begin{equation*}
\gamma(z) = \sum_{y\in\mathfrak{B}} g_{\mathfrak{B}}(y,z)\big(\mu_0(y)-h(y)\big),
\end{equation*}
where $g_{\mathfrak{B}}$ is the Green function stopped at the set $\mathfrak{B}$. The function $\gamma$ has the following property
\begin{equation}
\label{eq:odometer_potential}
\Delta \gamma(z)  =  h(z)-\mu_0(z), \quad \text{for all } z \in \mathfrak{B}.
\end{equation}

\begin{lemma}
	\label{lem:calc_odometer}
	Let $\gamma$ be a function satisfying
	\eqref{eq:odometer_potential}, then the odometer function $u$ can be written as
	\begin{equation*}
	u \equiv (\gamma + s)\indicator_{\mathfrak{B}},
	\end{equation*}
	where $s = s_{-\gamma}^\mathfrak{B}$ is the least super-harmonic majorant of $-\gamma$, and
	\begin{equation*}
	\indicator_{\mathfrak{B}}(z) =
	\begin{cases}
	1, \quad & \text{if } z\in\mathfrak{B} \\
	0,       & \text{otherwise},
	\end{cases}
	\end{equation*}
	is the indicator function of the set $\mathfrak{B}$.
\end{lemma}
\begin{proof}
	First we show that the odometer can be expressed in terms
	of $\gamma$ and $\mathfrak{B}$.
	By \eqref{eq:odometer_potential}, we know that
	$\Delta(u-\gamma)(z) \leq 0$  for $z\in \mathfrak{B}$.
	Therefore, $u-\gamma$ is super-harmonic on $\mathfrak{B}$.
	Also $u$ is nonnegative on $\mathfrak{B}$ and this implies that
	$u-\gamma \geq -\gamma$ on $\mathfrak{B}$. Therefore $u-\gamma$ is a super-harmonic majorant of
	$-\gamma$, which implies that $u\geq \gamma + s$ on the set $\mathfrak{B}$.
	
	In order to prove that $u-\gamma \leq s$, let us consider the function
	$s + \gamma - u$, which is super-harmonic on the sandpile cluster 
	$\mathcal{S}=\{x\in \SG: u(z)>0\}$, because, for all 
	$z\in\mathcal{S}$, one has
	\begin{equation*}
	\Delta(s+\gamma-u)(z) = \Delta s(z) \leq 0.
	\end{equation*}
Outside the sandpile cluster $\mathcal{S}$, $u(z) = 0$, and because $s$ is a majorant of $-\gamma$, 
we have $s + \gamma - u \geq 0$.
By the minimum principle for super-harmonic functions this inequality extends to the inside of
$S$, hence $u\leq \gamma + s$. Therefore, $u=\gamma+s$ on $\mathfrak{B}\supset S$.
\end{proof}

While Lemma \ref{lem:calc_odometer} can in principle be used to compute the odometer function, 
it is often difficult to use it practice, when working on state spaces, other than  $\ZZ^d$. In our particular case of $\SG$, we
guess the odometer function and then we prove that our guess is correct.
The next Lemma gives us a way to accomplish this.
\begin{lemma}
\label{lem:guess_odometer}
	Let $u_{\star}:\SG\to \RR_{\geq 0}$ be a function and let
	\begin{align*}
	A_{\star} &= \big\{z\in \SG:\: u_{\star}(z) > 0\big\} \\
	\mu_{\star} &= \mu_0 + \Delta u_{\star}.
	\end{align*}
	If $A_{\star}$ is finite, $\mu_{\star}(z) = h(z)$ for all $z\in A_{\star}$ and $\mu_{\star} \leq h$ then $u = u_{\star}$.
\end{lemma}
\begin{proof}
	Choose a finite set $\mathfrak{B}$ such that $A_{\star} \subset \mathfrak{B}$ and a function
	$\gamma: \SG\to\RR$ satisfying \eqref{eq:odometer_potential}. We can then use the same argument as in
	the proof of Lemma \ref{lem:calc_odometer} to show that $u_{\star} \equiv (\gamma + s^{\mathfrak{B}}_{-\gamma} )\indicator_{\mathfrak{B}}$. By the Abelian property all such
	representations give the same function. Hence $u_{\star} = u$.
\end{proof}

Note that \textsc{Levine and Friedrich} \cite[Theorem 1]{friedrich_levine_2013}
used a similar approach to prove that a given function is equal to the
odometer function of a rotor-router aggregation process (see also \cite{kager_levine_2010,huss_sava_2011} where this
technique was also applied).

As an easy consequence we can interpret the stopped Green function as
the odometer function of a special divisible sandpile. 

\begin{corollary}
\label{cor:green_sandpile} Let $A\subset\SG$ be a finite. For $n > 0$ let $\mu_0(x) = \indicator_{A}(x) + n \delta_o(x)$ be the initial
sand configuration of a divisible sandpile with height function
\begin{align*}
h(x) = \begin{cases}
1 & \quad \text{for } x\in A, \\
\infty & \quad \text{for } x\not\in A.
\end{cases}
\end{align*}
The odometer function of this process is then given as 
$u(x) = n\cdot g_A(o,x)$. The limit sand distribution is equal to 
\begin{align*}
\mu(x) = \begin{cases}
 1 & \quad\text{for } x \in A, \\
 n\cdot\nu(x) & \quad\text{for } x \in \partial A, \\
 0            & \quad\text{otherwise},
\end{cases}
\end{align*}
where $\nu(x)$ is the harmonic measure of the set $A$.
\end{corollary}
\begin{proof}
The statement follows directly from Lemma \ref{lem:guess_odometer} together with the fact that $\Delta g_A(o,x) = -\delta_o(x)$, for all $x\in A$.
\end{proof}

\section{The sandpile cluster on the Sierpinski gasket}
\label{sec:sierpinski_gasket}

First of all, a simple combinatorial fact which involves cardinality of balls and their boundaries in $\SG$ will be needed.
\begin{lemma}
\label{lem:mass_second_wave}
For all $n\geq 1$, the following holds
$$\abs{B_n} + \abs{\partial B_n} = \abs{B_{n+1}} - \frac{1}{2}\abs{\partial B_{n+1}}.$$
\end{lemma}
\begin{proof}
We distinguish two cases depending on the parity of $n$.

\emph{Case 1:} If $n$ is even, then $B_{n+1} \setminus B_n = \partial B_{n+1}$.
	Moreover every vertex of $\partial B_n$ is connected to exactly two vertices of $\partial B_{n+1}$, which gives $2\abs{\partial{B_n}} = \abs{\partial{B_{n+1}}}$.
	Thus
	\begin{equation*}
	\abs{B_{n+1}} - \abs{B_n} = \abs{\partial B_{n+1}} = 
	\abs{\partial B_n} + 1/2\abs{\partial{B_{n+1}}},
	\end{equation*}
	which proves the claim.
	
	\emph{Case 2:} If $n$ is odd, then $B_{n+1}\setminus B_{n-1}$ is a disjoint
	union of $k$ isomorphic connected graphs, see Figure \ref{fig:second_wave}. Let $C_i$, $i=1,\ldots,k$ be the connected components of
	$B_{n+1}\setminus B_{n-1}$ and let $B^i_n = B_n\cap C^i$ and $\partial B^i_n = \partial B_n\cap C^i$, for all $i=1,\ldots,k$.
Because the $C^i$ are isomorphic it suffices to prove the relation
	\begin{equation}
	\label{eq:second_wave_restricted}
	\abs{B^i_n} + \abs{\partial B^i_n} = \abs{B^i_{n+1}} - 1/2\abs{\partial B^i_{n+1}},
	\end{equation}
	for any $i\in\{1,\ldots,k\}$. It is easy to see that
	 $\abs{\partial B^i_{n+1}} = 2$ and $\abs{B_{n+1}^i}-\abs{B_{n}^i}=\abs{\partial B_n^i}+1$, which
proves \eqref{eq:second_wave_restricted}.
\end{proof}
\begin{figure}
	\centering
	\input{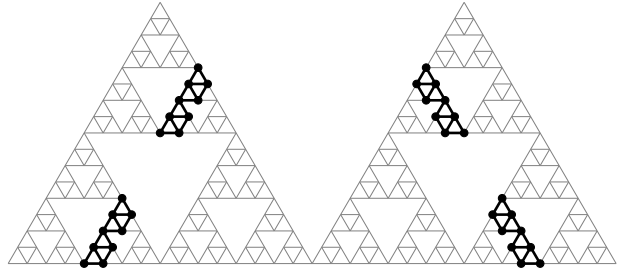}
	\caption{\label{fig:second_wave} The connected components of $B_{n+1} \setminus B_{n-1}$, for $n$ odd.}
\end{figure}

\begin{theorem}
\label{thm:odometer_bn} For every integer $n$ let $\mu_n$ be the limit sand distribution of the divisible sandpile on $\SG$ with initial mass distribution
$\mu_{n,0} \equiv b_n \delta_{(0,0)}$, where $b_n = \abs{B_n} - 1/2\abs{\partial B_n}$. Then $\mu_n$ is given by
\begin{align*}
\mu_{n}(z) = \begin{cases}
             1 & \text{if } z \in B_n \setminus \partial{B_n} \\
             1/2 & \text{if } z \in \partial{B_n}\\
             0 & \text{if } z \not\in B_n,
             \end{cases}
\end{align*}
and the corresponding sandpile cluster is $B_{n-1}$.
\end{theorem}

\begin{proof}
The proof goes by induction over $n$. For the base case we have $b_1 = 3$, thus after one single toppling of the origin we already reach the limit sand configuration with mass $1$ at the origin and mass $1/2$ at all neighbors of the origin.

Now assume that the statement of the theorem is true for some $n$ and denote by $u_n$
the limit odometer function of the sandpile with initial sand distribution $\mu_{n,0}$.
In the inductive step we want to construct the odometer function $u_{n+1}$ of the
sandpile with initial mass $b_{n+1}$ at the origin using the odometer function $u_n$.
We will accomplish this by splitting the sandpile topplings into three separate waves. First we send mass $b_n$ from the origin, and then the remaining mass $b_{n+1} - b_{n} = 3/2\abs{\partial B_n}$ (by Lemma \ref{lem:mass_second_wave}) will be send in the last two waves.

\emph{1st wave:} The first wave is just the sandpile with initial distribution
$\mu_{n,0}$ and sandpile height function $h(x)=1$. By the induction hypothesis the
odometer $u^{\bm{(1)}}$ of this first wave is equal to $u_n$ and the final mass distribution
$\mu^{\bm{(1)}} = \mu_n$.

\emph{2nd wave:} For the second wave we start with the final sand configuration of
the first wave $\mu^{\bm{(1)}}$ and add the remaining mass $3/2\abs{\partial B_n}$ at the origin.
For this second wave we only topple sites that where fully occupied (i.e. have mass 1) during the first wave. That is, we look at the divisible sandpile with initial mass configuration
\begin{align*}
\mu_0^{\bm{(2)}} = \mu^{\bm{(1)}} + 3/2\abs{\partial B_n} \delta_o,
\end{align*}
and sandpile height function 
\begin{align*}
	h(x) = \begin{cases}
             1 \quad& \text{for } x\in B_n \setminus \partial B_n, \\
             \infty & \text{otherwise}.
   	       \end{cases}
\end{align*}
Since by the induction hypothesis $\mu^{\bm{(1)}}$ is equal to $1$ on
$B_n \setminus\partial B_n$ we can apply Corollary \ref{cor:green_sandpile} and we
get for the odometer function of the second wave
\begin{align*}
u^{\bm{(2)}}(x) = 3/2\abs{\partial B_n} g_{B_n \setminus\partial B_n}(o,x),
\end{align*}
where $g_{B_n \setminus\partial B_n}$ is the Green function stopped at the set $B_n \setminus\partial B_n$. Moreover the final sand distribution after the second wave of topplings is
given by
\begin{align*}
\mu^{\bm{(2)}}(x) = \mu^{\bm{(1)}}(x) + 3/2\abs{\partial B_n}\nu(x),
\end{align*}
where $\nu(x) = \Pb_o[X(\tau_{B_n \setminus\partial B_n}) = x]$ is the harmonic measure of the set $B_n \setminus\partial B_n$ with the simple
random walk started at the origin.
The support of $\nu$ is exactly the set $\partial B_n$. Moreover by the symmetry of
the Sierpinski gasket graph it is clear that $\nu$ is uniform on the set $\partial B_n$, i.e.,
\begin{align*}
\nu(x) = \frac{\indicator_{\partial B_n}(x)}{\abs{\partial B_n}},
\end{align*}
which implies
\begin{align*}
\mu^{\bm{(2)}}(x) = \begin{cases}
   1\quad &\text{for } x\in B_n\setminus \partial{B_n},\\
   2      &\text{for } x\in \partial{B_n},\\
   0      &\text{otherwise}.
   \end{cases}
\end{align*}

\begin{figure}
\centering
\includegraphics[width=\textwidth]{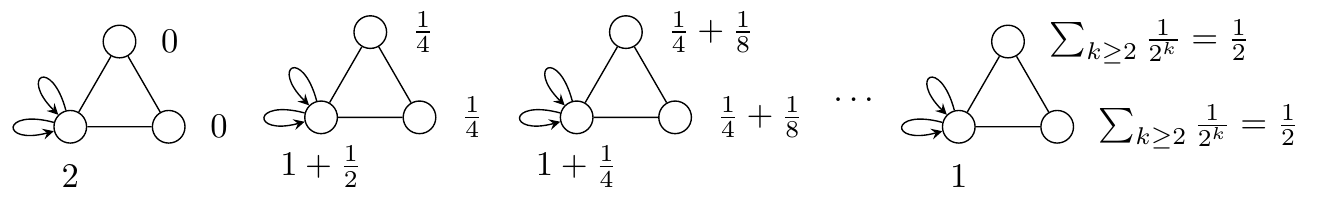}
\caption{\label{fig:sandpile_triangle}Steps of a divisible sandpile on a triangle with two added loops at the origin, starting with a mass of $2$ at the origin. The limit odometer at the origin is equal to $u(o) = 1+\frac{1}{2} + \frac{1}{4} + \frac{1}{8} + \dots = 2$.}
\end{figure}

\emph{3rd wave:} For the 3rd wave we start with the final mass distribution of the
second wave, that is $\mu_0^{\bm{(3)}} = \mu^{\bm{(2)}}$, and we use again the usual height function $h(x)\equiv 1$.
The situation at the start of the $3$rd wave is depicted in
 Figure \ref{fig:third_wave}. Each of the outer small triangles behaves like the gadget depicted in Figure \ref{fig:sandpile_triangle}. Since the gray area is already filled, all mass that is sent to the inside has to come out again eventually. By symmetry, the amount of mass sent out to each boundary point will be the same, thus the whole interior has the same effect as adding two loops to each boundary point in $\partial B_n$. Since no more mass can accumulate in the interior, the odometer function in $B_n\setminus\partial B_n$ increases during the 3rd wave by a harmonic function
which  is equal to $2$ at all the boundary points $\partial B_n$. It follows that the odometer function of the 3rd wave is given by
\begin{align*}
u^{\bm{(3)}}(x) = 2 \indicator_{\B_n}(x).
\end{align*}

\begin{figure}[b]
\centering
\includegraphics{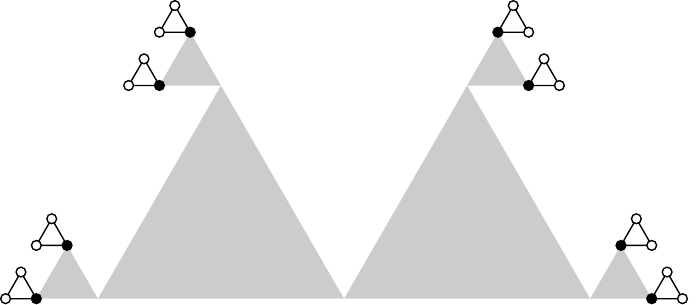}
\caption{\label{fig:third_wave}The sandpile at the start of the 3rd wave. The gray shaded region represents $B_n\setminus\partial B_n$ which is completely filled, i.e., $\mu^{(2)} \equiv 1$. The black dots $\bullet$ are the boundary points in $\partial B_n$ with sand height $2$. The white dots $\circ$ contain no sand.}
\end{figure}

To finish the argument we have to show that the sum of the odometers of the three waves $u^\star = u^{\bm{(1)}} + u^{\bm{(2)}} + u^{\bm{(3)}}$ is equal to $u_{n+1}$.
For this we apply Lemma \ref{lem:guess_odometer} to $u^\star$.
By the induction hypothesis we have
\begin{align*}
\Delta u^{\bm{(1)}}(x) &= 
\begin{cases}
1 - b_n\delta_o(x)\quad& \text{if } x \in B_n \setminus \partial{B_n}, \\
1/2 & \text{if } x \in \partial{B_n},\\
0 & \text{if } x \not\in B_n.
\end{cases}
\intertext{For the odometers of the second and third wave we get}
\Delta u^{\bm{(2)}}(x) &= 
\begin{cases}
-3/2\abs{\partial B_n} \quad&\text{if } x = o, \\
3/2 & \text{if } x \in \partial{B_n},\\
0 & \text{otherwise}
\end{cases}
\intertext{and}
\Delta u^{\bm{(3)}}(x) &= \begin{cases}
-1 \quad &\text{if } x \in \partial B_n, \\
0        &\text{if } x \in B_n \setminus \partial B_n, \\
1/2      &\text{if } x \in \partial B_{n+1}, \\
1        &\text{if } x \in \partial_I B_{n+1}, \\
0        &\text{otherwise}. 
\end{cases}
\intertext{Finally it follows by the linearity of the Laplacian that}
\Delta u^\star(x) &= \begin{cases}
1- (b_n + 3/2\abs{\partial B_n})\delta_o \quad&\text{if } x \in B_n,\\
1                                             &\text{if } x \in \partial B_n,\\
1                                             &\text{if } x \in \partial_I B_{n+1},\\
1/2                                           &\text{if } x \in \partial B_{n+1},\\
0                                             &\text{otherwise}
\end{cases} \\
&= \begin{cases}
1- b_{n+1}\delta_o \quad&\text{if } x \in B_{n+1}\setminus\partial B_{n+1}\\\
1/2                                           &\text{if } x \in \partial B_{n+1}\\
0                                             &\text{otherwise}
\end{cases} \\
&= \mu_{n+1} - \mu_{n+1,0}.
\end{align*}
Moreover $\{x\in\SG: u^\star(x) > 0\} = B_n$. Thus Lemma \ref{lem:guess_odometer}
implies that $u^\star = u_{n+1}$, which finishes the inductive step.
\end{proof}
The main result Theorem \ref{thm:sandpile_shape_thm} is just an application of the previous result.
\begin{proof}[Proof of Theorem \ref{thm:sandpile_shape_thm}.]
The proof proceeds by the same wave argument as used in the proof of Theorem \ref{thm:odometer_bn}.
In the first wave start with mass $b_n$ at the origin. By Theorem \ref{thm:odometer_bn}
the sandpile after this wave has exactly the shape $B_{n-1}$. For the second wave add
the remaining mass $m - b_n < 3/2\abs{\partial B_n}$ at the origin. For the third wave
there will be less then 2 mass at each vertex of $\partial B_n$. Hence the final sandpile
cluster after the third wave will be a subset of $B_{n}$.
\end{proof}

\section{Functions with constant Laplacian}
\label{sec:constant_laplacian}

For the proof of convergence of the divisible sandpile we have assumed the
existence of a function $\ell:\SG\to\RR$ with $\Delta \ell(x) = 1$ for
all $x\in\SG$. In this section we construct such a function with
particularly nice properties and show how is it connected with the
divisible sandpile model.

Recall that we are working with the representation of the $\SG$, as given in the left part of Figure \ref{fig:embedding}.
\begin{definition}
A subset $T$ of  $\SG$ is called a \emph{proper triangle} of size $2^k$, if the subgraph induced by $T$ is isomorphic to $V_k^+$.
If $T$ is a proper triangle, its extremal points $\partial T$ are
either of the form $\partial T =\big\{a, a + (2^k,0), a + (0,2^k)\big\}$ for some $a\in\SG^+$, or of the form $\partial T =\big\{a, a - (2^k,0), a - (0,2^k)\big\}$ for some $a\in\SG^-$.
\end{definition}

\begin{definition}
	Let $T$ be a proper triangle of size $2^k$ in $\SG$ for some $k\geq 1$,
	and let $\partial T = \big\{a, b, c)\big\}$,
	for some  $a,b,c\in \SG^+$. Without loss of generality we can assume $b = a + (2^k,0)$ and $c = a +(0,2^k)$.
	The \emph{midpoints} of $T$ are then given by
	\begin{align*}
	A &= (b+c)/2 = a + (2^{k-1}, 2^{k-1}) \\
	B &= (a+c)/2 = a + (0, 2^{k-1}),\\
	C &=(a+b)/2  = a + (2^{k-1},0).
	\end{align*}
For proper triangles which are subsets of $\SG^-$ the midpoints are
defined analogously. See Figure \ref{fig:proper_triangle} for a diagram of a proper triangle.	
\end{definition}

\begin{definition}
A subset $B$ of $\SG$ is called a \emph{proper ball} of size $2^k$, if the subgraph induced by $B$ is isomorphic to $V_k$.
See Figure \ref{fig:proper_triangle}.
\end{definition}

Next we define a simple function $\tilde{u}:\SG^+\to\RR$ which
has Laplacian $1$ everywhere except at the origin $o$.

\begin{definition}
	\label{def:utilde}
	Let $\tilde{u}:\SG^+\to\RR$ be the function defined by
	\begin{align*}
	\tilde{u}(x,0) &= 0 \text{ for all } x\geq 0, \\
	\tilde{u}(x,1) &= 2 \text{ for all } x\geq 0 \text{ s.t. } (x,1)\in\SG^+
	\end{align*}
	and $\Delta \tilde{u}(x,y) = 1$, for all $x,y\in\SG^+$ with $y\geq 1$.
\end{definition}

\begin{figure}[t]
	\centering
	\input{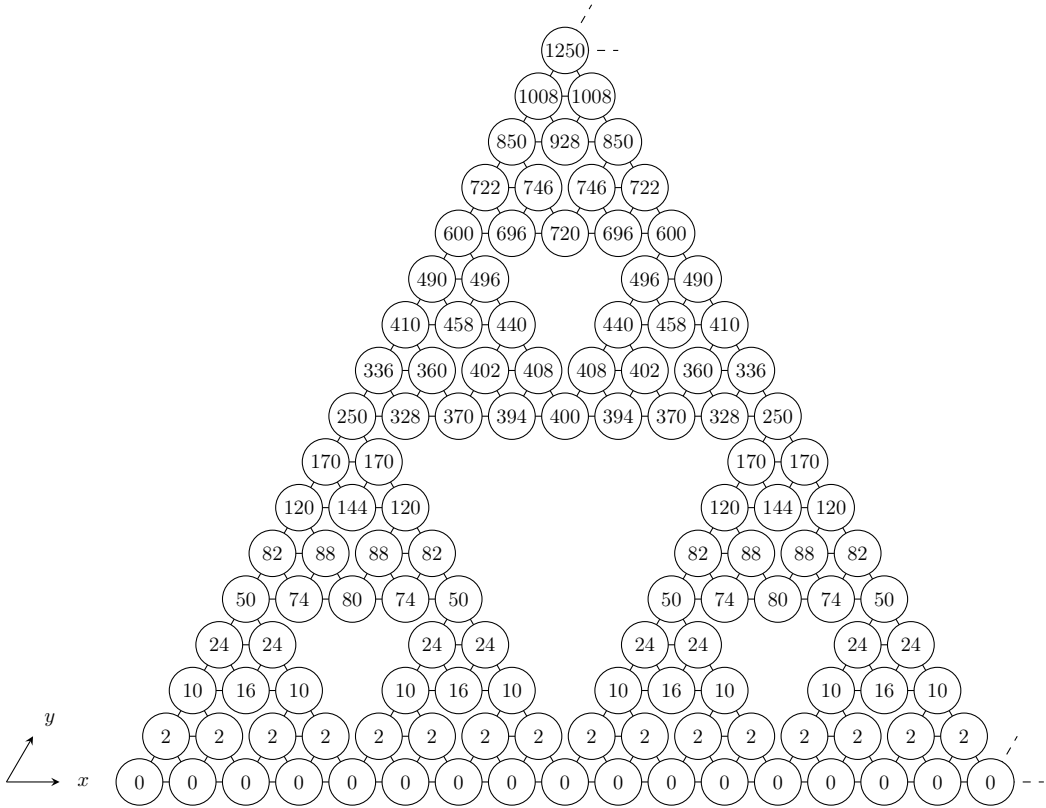}
	\caption{\label{fig:utilde} The function $\tilde{u}$.}
\end{figure}

\begin{remark}
	It is easy to check that the definition implies that $\Delta \tilde{u}(x,0) = 1$ for all $x\geq 1$.
\end{remark}
A priori, it is not clear whether such a function $\tilde{u}$ is well defined, and if it is unique. This is what we prove next. In the proof of the next result we are going to use a {\it generalized $\frac{1}{5}-\frac{2}{5}$ rule for functions with constant Laplacian} in Sierpinski gasket graphs. In the fractal community, the 
$\frac{1}{5}-\frac{2}{5}$ rule for harmonic functions on the gasket $\SG$ is well-known, but we need it in a more general setting. The proof of the {\it generalized $\frac{1}{5}-\frac{2}{5}$ rule for functions with constant Laplacian} will be postponed for the Appendix, in Theorem \ref{thm:1/5-2/5rule}. 

\begin{theorem}
	\label{thm:utilde_uniqueness}
	The function $\tilde{u}$ defined in Definition \ref{def:utilde} is unique.
	Moreover $\tilde{u}(0,2^k) = 2\cdot 5^k$ for all $k\geq 0$.
\end{theorem}
\begin{proof}
	The set $V_k^+$ is a proper triangle of size $2^k$ in $\SG^+$ with
	extremal points 
	\begin{equation*}
	\partial V_k^+ = \big\{ (0,0), (2^k,0), (0,2^k) \big\}.
	\end{equation*}
By definition $\tilde{u}(0,0) = \tilde{u}(2^k,0) = 0$.
Applying  Theorem
	\ref{thm:1/5-2/5rule}\eqref{thm:1/5-2/5rule:a}
	to the proper triangle $V_k^+$ gives 
	\begin{equation*}
	\tilde{u}(2^{k-1},0) = \frac{1}{5}\big( 2 \tilde{u}(2^k,0) + 2\tilde{u}(0,0) +\tilde{u}(0, 2^k)\big) - 2\cdot 5^{k-1},
	\end{equation*}
which together with $\tilde{u}(2^{k-1},0) = 0$, implies $\tilde{u}(0,2^k) = 2\cdot 5^k.$
	
Again by the {\it generalized $\frac{1}{5}-\frac{2}{5}$ rule} in Theorem \ref{thm:1/5-2/5rule}, the values of $\tilde{u}$ on any proper
	triangle $T$ are uniquely determined by its values on the extremal points
	$\partial T$. Hence the existence and uniqueness of $\tilde{u}$ follows.
\end{proof}

See Figure \ref{fig:utilde} for a plot of the function $\tilde{u}$. Note that when we extend $\tilde{u}$ to the whole of $\SG$ by
reflection at the origin we get a function with Laplacian equal to $1$ everywhere.
That is
\begin{equation*}
\ell(x,y) = \tilde{u}(\abs{x},\abs{y}),
\end{equation*}
is a function with Laplacian 1 globally, as  needed in Section \ref{subsec:sandpile_conv}. While
$\ell$ is only used as a technical tool in Section \ref{subsec:sandpile_conv} we will see below that the function $\tilde{u}$ has
a much deeper link to the divisible sandpile on the Sierpinski gasket.
We prove next some properties of $\tilde{u}$, in particular
that it is integer valued and non-negative.

\begin{lemma}\label{lem:tildeu_divisibility}
Let $T$ be a proper triangle of size $2^k$, $k\geq 1$,
with extremal points $\partial T = \big\{a,b,c\big\}$ and midpoints $M = \big\{A,B,C\big\}$. Fix $m\in\ZZ$ and let $h$ be a function on $\SG$ with $\Delta h(z) = m$, for all $z\in T$. Assume the value of $h(z)$ is divisible by $2\cdot 5^k$ for all $z\in\partial T$. Then $h(w)$ is divisible by $2\cdot 5^{k-1}$, for all $w \in M$.
\end{lemma}
\begin{proof}
By assumption $h(a) = 2\cdot 5^k\tilde{a}$, $h(b) = 2\cdot 5^k\tilde{b}$ and $h(c) = 2\cdot 5^k\tilde{c}$ for some $\tilde{a}, \tilde{b}, \tilde{c} \in \ZZ$. By Theorem \ref{thm:1/5-2/5rule}\eqref{thm:1/5-2/5rule:a}, we can then determine the value of $h$ at midpoints:
	\begin{align*}
	h(A) &= \frac{1}{5}\big(2\cdot 5^k\tilde{a} + 4\cdot 5^k\tilde{b} + 4\cdot 5^k\tilde{c}\big) - 2\cdot 5^{k-1} m \\
	&= 2\cdot 5^{k-1} \big(\tilde{a} + 2\tilde{b}+2\tilde{c} - m\big).
	\end{align*}
	Hence $h(A)$ is divisible by $2\cdot 5^{k-1}$. The divisibility of $h(B)$ and
	$h(C)$ follows by symmetry.
\end{proof}

\begin{theorem}
Let $T$ be an arbitrary proper triangle of size $2^k$, then for all extremal
points $z\in\partial T$, $\tilde{u}(z)\in\ZZ$ and is divisible by $2\cdot 5^k$.
\end{theorem}
\begin{proof}
Let $T$ be a proper triangle of size $2^k$, and  let $l\geq 1$ be the smallest
	integer such that $T$ is a subset of $V_l^+$. By the definition of $\tilde{u}$ and Theorem \ref{thm:utilde_uniqueness} the value of $\tilde{u}$ at all extremal points of $V_l^+$ is divisible by $2\cdot 5^l$. The values of $\tilde{u}$ at the extremal points of $T$ can be computed by applying 
Theorem \ref{thm:1/5-2/5rule}\ref{thm:1/5-2/5rule:a} recursively
at most $l-k$-times to the extremal points of $V_k^+$. The claim
follows then from Lemma \ref{lem:tildeu_divisibility}.
\end{proof}
As an immediate consequence we get the following.
\begin{corollary}
	$\tilde{u}(z) \in 2\ZZ$ for all $z\in\SG^+$.
\end{corollary}

\begin{theorem}
	\label{thm:utilde_positive}
The function $\tilde{u}$ is non-negative on all of $\SG^+$. Moreover $\tilde{u}(x,y)=0$ if and only if $y=0$.
\end{theorem}
\begin{proof}
	Let $a,b,c$ be the extremal points of a proper triangle of size $2^k$, and
	let $A,B,C$ be the midpoints of this triangle. By Lemma \ref{lem:tildeu_divisibility} there exists integers $\tilde{a}, \tilde{b}$ and
	$\tilde{c}$ such that $\tilde{u}(a) = 2\cdot 5^k\tilde{a}$, $\tilde{u}(b) = 2\cdot 5^k\tilde{b}$ and $\tilde{u}(c) = 2\cdot 5^k\tilde{c}$.
	Assume $\tilde{a},\tilde{b},\tilde{c}\geq 0$. It then follows again by
Theorem \ref{thm:1/5-2/5rule}\ref{thm:1/5-2/5rule:a} that
	\begin{equation*}
	\tilde{u}(A) = 2\cdot 5^{k-1}\big(\tilde{a} + 2\tilde{b} + 2\tilde{c} - 1\big).
	\end{equation*}
	Then 
$\tilde{u}(A)<0$ if and only if $\tilde{a}=\tilde{b}=\tilde{c}= 0$, and
	$\tilde{u}(A)=0$ if and only if $\tilde{a}=1$ and $\tilde{b}=\tilde{c}= 0$.
	As in Lemma \ref{lem:tildeu_divisibility} we can compute all values of
	$\tilde{u}$ with the {\it generalized $\frac{1}{5}-\frac{2}{5}$ rule} starting from a set $V_k^+$, for some $k\geq 1$. Since we have already show in Theorem \ref{thm:utilde_uniqueness} that $\tilde{u}(0,2^k) > 0$, it follows by induction that $\tilde{u}(x,y) > 0$
	for all $(x,y) \in \SG^+$ with $y>0$.
\end{proof}

\subsection{Explicit construction of the odometer function}
In section we will show that odometer of the divisible sandpile
with initial sand configuration $\mu(x) = 3^{k+1}\delta_o(x)$ is
essentially given by an affine transformation of the function
$\tilde{u}$. For this, we need to compute some particular values of $\tilde{u}$ explicitly.

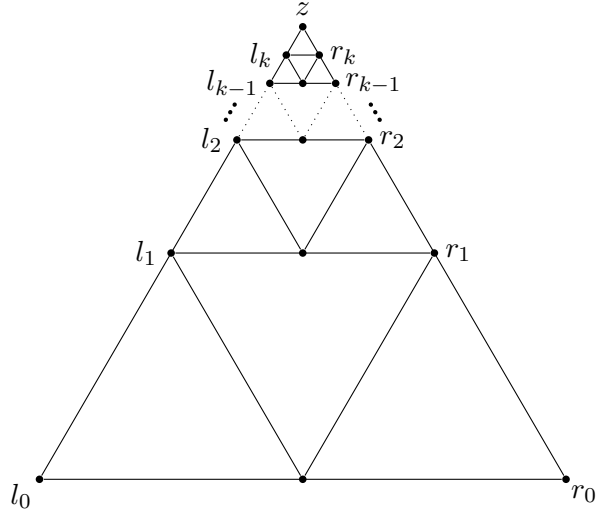
\begin{figure}[t]
\centering
\begin{tikzpicture}[scale=0.5,
every node/.style={fill, circle, inner sep = 1pt}]

\node[label=90:$z$] (z) at (120-30:8cm) {};
\node[label=210:$l_0$] (l0) at (2*120-30:8cm) {};
\node[label=-10:$r_0$] (r0) at (3*120-30:8cm) {};

\node[label=180:$l_1$] (l1) at ($(l0)!0.5!(z)$) {};
\coordinate (wl2) at ($(l1)!0.5!(z)$) {};
\node[fill = white, inner sep = 6pt, circle] (ldots0) at ($(wl2) - (0.6,0)$) {};
\node[label=180:$l_2$] (l2) at ($(l1)!0.5!(z)$) {};
\coordinate (wl3) at ($(l2)!0.5!(z)$) {};
\node[fill = white, inner sep = 6pt, circle] (ldots1) at ($(wl3) - (0.6,0)$) {};
\node[label=180:$l_{k-1}$] (l3) at ($(l2)!0.5!(z)$) {};
\node[label=180:$l_{k}$] (l4) at ($(l3)!0.5!(z)$) {};

\node[label=0:$r_1$] (r1) at ($(r0)!0.5!(z)$) {};
\coordinate (wr2) at ($(r1)!0.5!(z)$) {};
\node[fill = white, inner sep = 6pt, circle] (rdots0) at ($(wr2) + (0.6,0)$) {};
\node[label=0:$r_2$] (r2) at ($(r1)!0.5!(z)$) {};
\coordinate (wr3) at ($(r2)!0.5!(z)$) {};
\node[fill = white, inner sep = 6pt, circle] (rdots1) at ($(wr3) + (0.6,0)$) {};
\node[label=0:$r_{k-1}$] (r3) at ($(r2)!0.5!(z)$) {};
\node[label=0:$r_{k}$] (r4) at ($(r3)!0.5!(z)$) {};

\node (m0) at ($(l0)!0.5!(r0)$) {};
\node (m1) at ($(l1)!0.5!(r1)$) {};
\node (m2) at ($(l2)!0.5!(r2)$) {};
\node (m3) at ($(l3)!0.5!(r3)$) {};

\draw[line cap = round, dash pattern=on 0pt off 2\pgflinewidth, line width=1.6pt] (ldots0) -- (ldots1);
\draw[line cap = round, dash pattern=on 0pt off 2\pgflinewidth, line width=1.6pt] (rdots0) -- (rdots1);

\draw (l1) -- (r1) -- (m0) -- (l1);
\draw (l2) -- (m1) -- (r2) -- (l2);
\draw[dotted] (l2) -- (l3) -- (m2) -- (r3) -- (r2);
\draw (l3) -- (r3) -- (z) -- (l3);

\draw (l4) -- (m3) -- (r4) -- (l4);
\draw (l2) -- (l0) -- (r0) -- (r2);

\end{tikzpicture}
\caption{\label{fig:l_k} The construction used in Lemma \ref{lem:l_k}.}
\end{figure}
\begin{lemma}
	\label{lem:l_k}
For all $k\geq 1$ we have $\tilde{u}(0, 2^k-1) = \tilde{u}(1,2^k-1) = 1 - 3^{k+1} + 2\cdot 5^k$.
\end{lemma}
\begin{proof}
	Let $z = (0, 2^k)$ and define $l_i = (0, s_i)$ and $r_i = (2^k-s_i,s_i)$, where $s_i = \sum_{j=1}^i 2^{k-j}$, for all $i=0,\ldots,k$ (see Figure \ref{fig:l_k}). In particular we have
	$l_0 = (0,0)$, $r_0 = (0, 2^k)$, $l_k = (0, 2^k-1)$ and $r_k = (1, 2^k-1)$.
	By the  Theorem \ref{thm:1/5-2/5rule}\ref{thm:1/5-2/5rule:a} we get for all $i=1,\ldots,k$:
	\begin{align}
	\label{eq:utilde_recursion}
	\begin{aligned}
	\tilde{u}(l_i) &= \frac{1}{5}\big(\tilde{u}(r_{i-1}) + 2\tilde{u}(l_{i-1}) + 2 \tilde{u}(z) \big) - 2\cdot 5^{k-i}, \\
	\tilde{u}(r_i) &= \frac{1}{5}\big(\tilde{u}(l_{i-1}) + 2\tilde{u}(r_{i-1}) + 2 \tilde{u}(z) \big) - 2\cdot 5^{k-i}.
	\end{aligned}
	\end{align}
	Since $\tilde{u}(l_0) = \tilde{u}(r_0) = 0$ it follows by induction that
	$\tilde{u}(l_i) = \tilde{u}(r_i)$ for all $i = 0,\ldots, k$.
	This simplifies the recursion \eqref{eq:utilde_recursion} to
	\begin{align}
	\label{eq:utilde_recursion2}
	\begin{aligned}
	\tilde{u}(l_i) &= \frac{1}{5}\big( 3 \tilde{u}(l_{i-1}) + 2\tilde{u}(z)\big) - 2\cdot 5^{k-i} \\
	&= \frac{3}{5} \tilde{u}(l_{i-1}) + 4\cdot 5^{k-1} - 2\cdot 5^{k-i},
	\end{aligned}
	\end{align}
	where in the last line we used that $\tilde{u}(z) = 2\cdot 5^k$ by Theorem
	\ref{thm:utilde_uniqueness}. The linear recursion \eqref{eq:utilde_recursion2} has
	the explicit solution
	\begin{equation*}
	\tilde{u}(l_i) = 5^{k-i}\big(1- 3^{i+1} + 2\cdot 5^i\big).
	\end{equation*}
	Setting $i=k$ gives the result.
\end{proof}

Let $\psi_k: \ZZ^2 \to \ZZ^2$ be the function given by
$\psi_k(x,y) = (y,2^k-x-y)$, for all $k\geq 0$. We have $\psi_k(0,0) = (0,2^k)$, $\psi_k(0,2^k) = (2^k,0)$ and $\psi(2^k,0) = (0, 0)$. That is, $\psi_k$
maps $\partial V_k^+= \big\{(0,0), (2^k,0), (0,2^k)\big\}$ onto itself. Moreover it is easy to check that $\psi_k$ is bijective on $V_k^+$ and acts as a rotation by
$-120^\circ$ around the center of the biggest hole in $V_k^+$.

\begin{theorem}
	\label{thm:odometer_V_k}
	Let $u_k:\SG \to \RR_{\geq 0}$ be the odometer function of divisible sandpile on $\SG$ with initial mass distribution $\mu_0 \equiv 3^{k+1}\delta_{(0,0)}$. Then for
	all $k\in \NN_0$
	\begin{equation}
	\label{eq:u_vk} 
	u_k(x,y) = \begin{cases}
	(\tilde{u} \circ \psi_k)(\abs{x}, \abs{y}), \quad&\text{ if } (x,y)\in B_{2^k},\\
	0                                         ,  &\text{ otherwise}.
	\end{cases}
	\end{equation}
	Moreover the sandpile cluster $\mathcal{S} = \big\{z\in\SG:\: u_k(z)>0\big\} = B_{2^k-1}$.
\end{theorem}
\begin{proof}
We check that the requirements in Lemma \ref{lem:guess_odometer}
are fulfilled for the  function $u_k(x,y)$ defined in \eqref{eq:u_vk}.
By construction and Theorem \ref{thm:utilde_positive}, $u_k(z) > 0$ if and only if $z\in B_{2^k-1}$. We need to check that $\Delta u_k(z) = 1 - \mu_0$ for all
	$z\in B_{2^k-1}$. For all $z\in B_{2^k-1}\setminus\{(0,0)\}$ is follows directly since $\Delta\tilde{u}(x,y) = 1$ for all $y\geq 1$ by definition.
For $z\in S_{2^k}$, we have $\Delta u(z) \in \{1/2, 1\}$.
	
Thus we only need to calculate the Laplacian of $u_k$ at the origin. Since
$\psi_k(0,0) = (0,2^k)$ we have
	$u_k(0,0) = \tilde{u}(0,2^k) = 2\cdot 5^k$ by Theorem \ref{thm:utilde_uniqueness}.
	For $z\sim(0,0)$ we have $\psi_k(z) \in \{(1,2^k-1), (0,2^k-1)\}$. Hence
	$u_k(z) = 1 - 3^{k+1} + 2\cdot 5^k$ by Lemma \ref{lem:l_k} for all $z\sim (0,0)$.
	Thus $\Delta u_k(0,0) = \frac{1}{4}\sum_{z\sim (0,0)} u_k(z) - u_k(0,0) = 
	1- 3^{k+1} + 2\cdot 5^k - 2\cdot 5^k = 1 - 3^{k+1}$, which completes the proof.
\end{proof}

\begin{remark}
Theorem \ref{thm:odometer_V_k} is a more explicit version of Theorem \ref{thm:odometer_bn} for the special case $n=2^k$.
\end{remark}

%
%


\section{Open questions}

The Sierpinski gasket graph $\SG$ is one of the simplest pre-fractals which has the property of being finitely ramified. This is very often used throughout the paper, especially when constructing explicitly the function with  Laplacian $1$ on $\SG$.  It might be interesting to prove a limit shape theorem for the divisible sandpile model on the Sierpinski carpet graph, 
which is infinitely ramified.  The Sierpinski carpet still has some special features (symmetry in particular), but is general enough so that one needs to develop more powerful techniques in order to analyze the behavior of the divisible sandpile model, model which turned out to be very helpful in proving limit shape theorems for the stochastic growth model {\it internal DLA}.

\begin{appendices}
\section{Generalized $\frac{1}{5}-\frac{2}{5}$ rule}
	
The following version of the {\it $\frac{1}{5}-\frac{2}{5}$ rule for harmonic functions} is probably known but since we did not find it in the literature in the form we need it here, we add a proof of this fact for completeness.
	
\begin{theorem}[$\frac{1}{5}-\frac{2}{5}$ rule for functions with constant Laplacian]
\label{thm:1/5-2/5rule}
The following two properties are true for all $k\in\NN$:
		\begin{enumerate}[(a)]
			\item\label{thm:1/5-2/5rule:a}
			Let $T$ be a proper triangle of size $2^k$ and extremal points
			$\partial T = \big\{a, b, c\big \}$, and midpoints $\big\{A,B,C\big\}$ as in Figure \ref{fig:proper_triangle}.
			Let $h:\SG \to\RR$ be a function such that $\Delta h(x) = m$ for
			all $x\in T$. Then the values of $h$ at the midpoints are given by
			\begin{align*}
			h(A) &= \frac{1}{5}\big(h(a) + 2 h(b) + 2 h(c)\big) - 2\cdot 5^{k-1}m, \\
			h(B) &= \frac{1}{5}\big(2 h(a) + h(b) + 2 h(c)\big) - 2\cdot 5^{k-1}m, \\
			h(C) &= \frac{1}{5}\big(2 h(a) + 2 h(b) + h(c)\big) - 2\cdot 5^{k-1}m. \\
			\end{align*}
			\item\label{thm:1/5-2/5rule:b}
			Let $B$ be a proper ball of size $2^{k-1}$ with center $D$
			and extremal points $\big\{d_1,d_2,d_3,d_4\big\}$ as in
			Figure \ref{fig:proper_triangle}.
			Let $h:\SG \to\RR$ be a function such that $\Delta h(x) = m$ for
			all $x\in B$. Then the value of $h$ at the center point is given by
			\begin{align*}
			h(D) = \frac{1}{4}\big(h(d_1) + h(d_2) + h(d_3) + h(d_4)\big) - 5^{k-1}m.
			\end{align*}
		\end{enumerate}
	\end{theorem}
	\begin{figure}[t]
		\centering
		\input{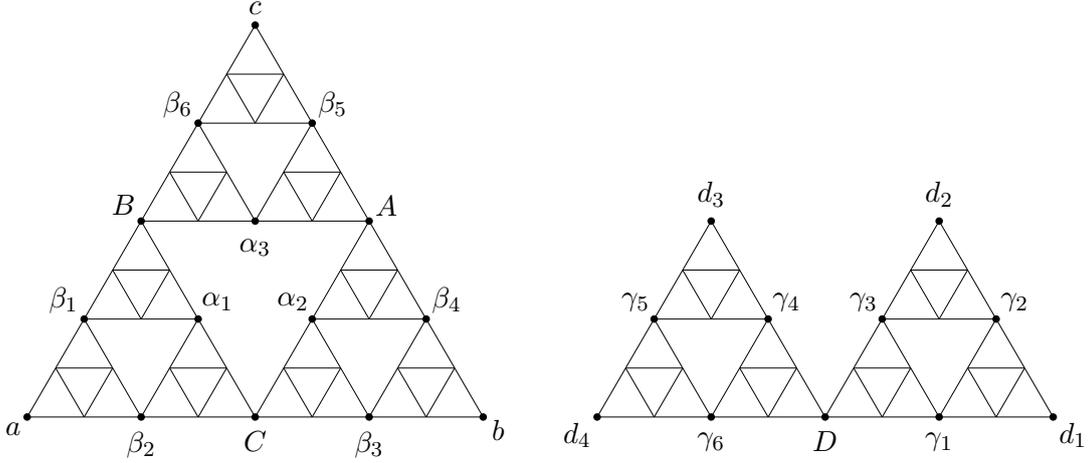} 
		\caption{\label{fig:proper_triangle} \textit{Left:} Proper triangle of size $2^3$, with extremal points $\{a,b,c\}$ and midpoints $\{A,B,C\}$. 
			\textit{Right:} Proper ball of size $2^2$ with center $D$ and extremal points $\{d_1,d_2,d_3,d_4\}$.} 
	\end{figure}
	\begin{proof}
		The proof goes by induction. First we consider the basis case $k=1$. To see relation \eqref{thm:1/5-2/5rule:a} note the the
		function values at the midpoints are related to the values at
		the extremal points, by the following linear equation:
		\begin{align*}
		\left(\begin{array}{rrr}
		-4 &  1 &  1 \\
		1 & -4 &  1 \\
		1 &  1 & -4
		\end{array}\right)
		\cdot
		\begin{pmatrix}
		h(A) \\ h(B) \\ h(C)
		\end{pmatrix}
		+
		\begin{pmatrix}
		0 & 1 & 1 \\
		1 & 0 & 1 \\
		0 & 1 & 1
		\end{pmatrix}
		\cdot
		\begin{pmatrix}
		h(a) \\ h(b) \\ h(c)
		\end{pmatrix}
		=
		\begin{pmatrix}
		4m \\ 4m \\ 4m
		\end{pmatrix}.
		\end{align*}
		Solving this equation gives relation \eqref{thm:1/5-2/5rule:a}.
		For $k=1$ relation \eqref{thm:1/5-2/5rule:b} is just $\Delta h(D) = m$, and is thus true by assumption.
		
		For the inductive step assume that both relations \eqref{thm:1/5-2/5rule:a} and \eqref{thm:1/5-2/5rule:b} are true for some $k\geq 1$.
A proper triangle $T$ of size $2^{k+1}$ consists of three proper triangles of size $2^k$ which pairwise share one extremal point.
		See the left hand side of Figure \ref{fig:proper_triangle}. For
		each of these smaller triangles we can use relation \eqref{thm:1/5-2/5rule:a} for the points $\alpha_1,\alpha_2, \alpha_3$ and $ \beta_1,\ldots,\beta_6$ in the notation of Figure \ref{fig:proper_triangle}.
Moreover note that points $\big\{b, c, B, C\big\}$ are the extremal
points of a proper ball of size $2^k$ with center $A$. Similarly
$T$ contains two more proper balls of size $2^k$ with center points
$B$ and $C$. We can apply relation \eqref{thm:1/5-2/5rule:b} to these three proper balls. Let 
		\begin{equation*}
		\mathbf{h}=\big(h(A), h(B), h(C), h(\alpha_1), \ldots, h(\alpha_3), h(\beta_1), \ldots, h(\beta_6)\big)^T
		\end{equation*} be the column vector of unknowns. This leads
		to a system of linear equations given by
		\begin{equation*}
		\mathbf{M} \cdot\mathbf{h} = \mathbf{Y}\cdot
		\big(h(a), h(b), h(c), m\big)^T,
		\end{equation*}
		where the matrices $\mathbf{M}$ and $\mathbf{Y}$ are given by
		\begin{align*}
		\mathbf{M}=\left(\begin{array}{rrr*{9}{>{\greytest}r}}
		0 & -2 & -2 & 5 & 0 & 0 & 0 & 0 & 0 & 0 & 0 & 0 \\
		-2 &  0 & -2 & 0 & 5 & 0 & 0 & 0 & 0 & 0 & 0 & 0 \\
		-2 & -2 &  0 & 0 & 0 & 5 & 0 & 0 & 0 & 0 & 0 & 0 \\
		0 & -2 & -1 & 0 & 0 & 0 & 5 & 0 & 0 & 0 & 0 & 0 \\
		0 & -1 & -2 & 0 & 0 & 0 & 0 & 5 & 0 & 0 & 0 & 0 \\
		-1 &  0 & -2 & 0 & 0 & 0 & 0 & 0 & 5 & 0 & 0 & 0 \\
		-2 &  0 & -1 & 0 & 0 & 0 & 0 & 0 & 0 & 5 & 0 & 0 \\
		-2 & -1 &  0 & 0 & 0 & 0 & 0 & 0 & 0 & 0 & 5 & 0 \\
		-1 & -2 &  0 & 0 & 0 & 0 & 0 & 0 & 0 & 0 & 0 & 5 \\
		4 & -1 & -1 & 0 & 0 & 0 & 0 & 0 & 0 & 0 & 0 & 0 \\
		-1 &  4 & -1 & 0 & 0 & 0 & 0 & 0 & 0 & 0 & 0 & 0 \\
		-1 & -1 &  4 & 0 & 0 & 0 & 0 & 0 & 0 & 0 & 0 & 0
		\end{array}\right) \text{ and }
		\mathbf{Y} = \left(\begin{array}{*{4}{>{\greytest}r}}
		1 & 0 & 0 & -2 \cdot 5^{k} \\
		0 & 1 & 0 & -2 \cdot 5^{k} \\
		0 & 0 & 1 & -2 \cdot 5^{k} \\
		2 & 0 & 0 & -2 \cdot 5^{k} \\
		2 & 0 & 0 & -2 \cdot 5^{k} \\
		0 & 2 & 0 & -2 \cdot 5^{k} \\
		0 & 2 & 0 & -2 \cdot 5^{k} \\
		0 & 0 & 2 & -2 \cdot 5^{k} \\
		0 & 0 & 2 & -2 \cdot 5^{k} \\
		0 & 1 & 1 & -4 \cdot 5^{k} \\
		1 & 0 & 1 & -4 \cdot 5^{k} \\
		1 & 1 & 0 & -4 \cdot 5^{k}
		\end{array}\right).
		\end{align*}
		
Then the first three lines of $\mathbf{h} = \mathbf{M}^{-1} \cdot\mathbf{Y}\cdot
		\big(h(a), h(b), h(c), m\big)^T$ give relation \eqref{thm:1/5-2/5rule:a} for $k+1$:
		\begin{align*}
		\begin{pmatrix}
		h(A) \\ h(B) \\ h(C)
		\end{pmatrix}
		= \frac{1}{5}
		\begin{pmatrix}
		1 & 2 & 2 & -2 \cdot 5^{k+1} \\
		2 & 1 & 2 & -2 \cdot 5^{k+1} \\
		2 & 2 & 1 & -2 \cdot 5^{k+1}
		\end{pmatrix} \cdot
		\begin{pmatrix}
		h(a) \\ h(b) \\ h(c) \\ m
		\end{pmatrix}
		\end{align*}
To prove the inductive step for relation \eqref{thm:1/5-2/5rule:b},
note that a proper ball of radius $2^k$ (see the right side of Figure \ref{fig:proper_triangle}), consists of two proper triangles
of size $2^k$ with extremal points $\big\{D, d_1, d_2\big\}$ resp.
$\big\{D, d_3, d_4\big\}$. Moreover $D$ is the center of a proper ball of size $2^{k-1}$ and extremal points $\big\{\gamma_1, \gamma_3,\gamma_4,\gamma_6\}$, in the notation of Figure \ref{fig:proper_triangle}.
		
Let $\bm{\widetilde{h}} = \big(h(D), h(\gamma_1), \ldots, h(\gamma_6)\big)^T$ be the vector of unknowns.
Applying the induction hypothesis to these subsets leads to the linear equation
		\begin{equation*}
		\bm{\widetilde{M}}\cdot\bm{\widetilde{h}} = \bm{\widetilde{Y}}\cdot
		\big(h(d_1),\ldots, h(d_4), m\big)^{T},
		\end{equation*}
		where the matrices $\bm{\widetilde{M}}$ and $\bm{\widetilde{Y}}$ are given by
		\begin{equation*}
		\bm{\widetilde{M}}=
		\left(\begin{array}{*{7}{>{\greytest}r}}
		4 & -1 & \phantom{-}0 & -1 & -1 & \phantom{-}0 & -1 \\
		-2 & 5 & 0 & 0 & 0 & 0 & 0 \\
		-1 & 0 & 5 & 0 & 0 & 0 & 0 \\
		-2 & 0 & 0 & 5 & 0 & 0 & 0 \\
		-2 & 0 & 0 & 0 & 5 & 0 & 0 \\
		-1 & 0 & 0 & 0 & 0 & 5 & 0 \\
		-2 & 0 & 0 & 0 & 0 & 0 & 5
		\end{array}\right)
		\text{ and }
		\bm{\widetilde{Y}} = 
		\left(\begin{array}{*{4}{>{\greytest}r}l}
		0 & 0 & 0 & 0 & -4 \cdot 5^{k - 1} \\
		2 & 1 & 0 & 0 & -2 \cdot 5^{k} \\
		2 & 2 & 0 & 0 & -2 \cdot 5^{k} \\
		1 & 2 & 0 & 0 & -2 \cdot 5^{k} \\
		0 & 0 & 2 & 1 & -2 \cdot 5^{k} \\
		0 & 0 & 2 & 2 & -2 \cdot 5^{k} \\
		0 & 0 & 1 & 2 & -2 \cdot 5^{k}
		\end{array}\right).
		\end{equation*}
		Then the first line of $\bm{\widetilde{h}} = \bm{\widetilde{M}}^{-1} \cdot\bm{\widetilde{Y}}\cdot
		\big(h(d_1), \ldots, h(d_4), m\big)^T$ gives relation \eqref{thm:1/5-2/5rule:b} for $k+1$:
		\begin{equation*}
		h(D) = 
		\left(\begin{array}{rrrrr}
		\dfrac{1}{4} & \dfrac{1}{4} & \dfrac{1}{4} & \dfrac{1}{4} &
		-5^{k}
		\end{array}\right)\cdot 
		\begin{pmatrix}
		h(d_1) \\
		\vdots \\
		h(d_4) \\
		m
		\end{pmatrix},
		\end{equation*}
and this finishes the proof.
\end{proof}
\end{appendices}

\paragraph{Acknowledgements.}
The research of Wilfried Huss was supported by the Austrian Science Fund (FWF): 
\href{http://pf.fwf.ac.at/en/research-in-practice/in-the-spotlight-schroedinger/list-of-schroedinger-fellows/2014/8523}{J3628-N26} and P25510-N26.
The research of Ecaterina Sava-Huss was supported by the Austrian Science Fund (FWF):
\href{http://pf.fwf.ac.at/en/research-in-practice/in-the-spotlight-schroedinger/list-of-schroedinger-fellows/2014/210946}{J3575-N26}.

\bibliography{sandpile-gasket}{}
\bibliographystyle{alpha_arxiv}

\textsc{Wilfried Huss, Institute of Discrete Mathematics, Graz University of Technology, 8010 Graz, Austria.}
\texttt{huss@math.tugraz.at} \\ \url{http://www.math.tugraz.at/~huss}

\textsc{Ecaterina Sava-Huss, Institute of Discrete Mathematics, Graz University of Technology, 8010 Graz, Austria.}
\texttt{sava-huss@tugraz.at}\\
\url{http://www.math.tugraz.at/~sava} 

\end{document}